\documentclass[12pt]{article}

\usepackage{amssymb}
\usepackage{amsthm}
\usepackage{amsmath}
\usepackage{float}
\usepackage[margin=2.5cm]{geometry}
\usepackage{makeidx}
\usepackage{amsfonts}
\usepackage{enumerate}

\usepackage{wrapfig}
\usepackage{here}
\usepackage[dvips]{graphicx}

\usepackage{tikz}

\theoremstyle{plain}
\newtheorem{theorem}{Theorem}[section]

\newtheorem{lemma}[theorem]{Lemma}
\newtheorem{corollary}[theorem]{Corollary}
\newtheorem{conjecture}[theorem]{Conjecture}
\newtheorem{proposition}[theorem]{Proposition}

\newtheorem{problem}{Problem}
\theoremstyle{definition}

\numberwithin{equation}{section}
\definecolor{blau}{rgb}{0.1,0.0,0.9}
\definecolor{gruen}{cmyk}{1.0,0.2,0.7,0.07}
\definecolor{mag}{cmyk}{0.0,0.9,0.3,0.0}

\newcommand{\black}{\color{black}}

\begin{document}

\title{Some bounds on the number of colors in interval and cyclic interval edge colorings of graphs}

\author{
%{\sl Armen S. Asratian}\thanks{{\it E-mail address:}
%armen.asratian@liu.se}  \\
%Department of Mathematics \\
%Link\"oping University \\
%SE-581 83 Link\"oping, Sweden
%%
%\and
%
{\sl Carl Johan Casselgren}\thanks{{\it E-mail address:}
carl.johan.casselgren@liu.se}  \\
Department of Mathematics \\
Link\"oping University \\
SE-581 83 Link\"oping, Sweden
\and
{\sl  Hrant H. Khachatrian}\thanks{{\it E-mail address:}
hrant@egern.net}\\
Department of Informatics \\ and Applied Mathematics \\
Yerevan State University \\
0025, Armenia
\and
{\sl Petros A. Petrosyan}\thanks{{\it E-mail address:}
pet\_petros@ipia.sci.am} \\
Department of Informatics \\ and Applied Mathematics \\
Yerevan State University \\
0025, Armenia
}

\maketitle

\bigskip
\noindent {\bf Abstract.} An \emph{interval $t$-coloring} of a
multigraph $G$ is a proper edge coloring with colors $1,\dots,t$
such that the colors on the edges incident to every vertex of $G$
are colored by consecutive colors. A \emph{cyclic interval $t$-coloring} 
of a
multigraph $G$ is a proper edge coloring with colors $1,\dots,t$
such that the colors on the edges incident to every vertex of $G$
are colored by consecutive colors, under the
condition that color $1$ is considered as consecutive to color $t$.
Denote by $w(G)$ ($w_{c}(G)$) and $W(G)$ ($W_{c}(G)$) the minimum
and maximum number of colors in a (cyclic) interval coloring of a
multigraph $G$, respectively. 
We present some new sharp bounds on $w(G)$ and
$W(G)$ for multigraphs $G$ satisfying various conditions. In particular,
we show that if $G$ is a $2$-connected multigraph with an interval
coloring, then $W(G)\leq 1+\left\lfloor \frac{\vert
V(G)\vert}{2}\right\rfloor(\Delta(G)-1)$. We also give several
results towards the general conjecture that $W_{c}(G)\leq \vert
V(G)\vert$ for any triangle-free graph $G$ with a cyclic interval
coloring; we establish that 
approximate versions of this conjecture hold 
for several families of graphs, and
we prove that the conjecture is true for graphs with maximum degree at most $4$.

\section{Introduction}\

   In this paper we consider graphs which are finite, undirected, and have no loops or multiple edges, and multigraphs which may contain multiple edges but no loops. We denote by $V(G)$ and $E(G)$ the sets of vertices 
and edges of a multigraph $G$, respectively, 
and by $\Delta(G)$ and $\delta(G)$ the maximum and minimum degrees 
of vertices in $G$, respectively. 
The terms and concepts that we do not define here can be found in \cite{AsratianHaggkvist,West}.

    An {\em interval coloring} (or {\em consecutive coloring})
    of a multigraph $G$ is a proper coloring of the edges by positive integers
    such that the colors on the edges incident to any vertex of $G$
    form an interval of integers, and no color class is empty.
 The notion of
interval colorings was introduced by Asratian and Kamalian
\cite{AsratianKamalian}
    (available in English as \cite{AsratianKamalian2}),
    motivated by the problem of finding compact school timetables, that is,
    timetables such that the lectures of each teacher and each class are
    scheduled at consecutive periods. 
		We denote by $\mathfrak{N}$ the set of all interval
		colorable multigraphs.
		
    All regular bipartite graphs have interval colorings, since they
    decompose into perfect matchings. Not every graph has an interval coloring,
    since a graph $G$ with an interval coloring must have
    a proper $\Delta(G)$-edge-coloring \cite{AsratianKamalian}.
    Sevastjanov \cite{Sevastjanov} proved that determining whether
    a bipartite graph has an interval coloring is $\mathcal{NP}$-complete.
    Nevertheless, trees \cite{KamalianPreprint},
    regular and complete bipartite graphs 
		\cite{KamalianPreprint,KamalianThesis,Hansen},
    grids \cite{timetabling4}, and
    outerplanar bipartite graphs \cite{timetabling3, Axenovich}
    all have interval colorings. Moreover, some families
    of $(a,b)$-biregular graphs have been proved to admit interval
    colorings \cite{HansonLotenToft,Hansen,
		CasselgrenToft,AsratianCasselgrenVandenWest,
		Pyatkin,YangLi},
    where a bipartite graph is {\em $(a,b)$-biregular} if 
		all vertices in one part have degree
    $a$ and all vertices in the other part have degree $b$.

    A {\em cyclic interval $t$-coloring} of a multigraph $G$ is a 
		proper $t$-edge-coloring $\alpha$
    such that for every vertex $v$ of $G$ the colors on the edges
    incident to $v$ either form an interval of
    integers or the set $\{1,\ldots,t\}\setminus \{\alpha(e): \text{$e$
    is incident to $v$}\}$ is an interval of integers,
		and no color class is empty.
    This notion was introduced by de Werra and Solot \cite{deWerraSolot}
		motivated by scheduling problems arising in flexible
		manufacturing systems, in particular the so-called 
		\emph{cylindrical open shop scheduling problem} \cite{deWerraSolot}. 
		%Note that any
		%interval coloring of a multigraph is also a cyclic interval coloring. 
		We denote by $\mathfrak{N}_{c}$ the set of all cyclically 
		interval colorable multigraphs.

    Cyclic interval colorings are studied in e.g.
    \cite{KubaleNadolski,Nadolski,PetrosyanMkhit,AsratianCasselgrenPetrosyan}.
    In particular, the general question of determining whether 
		a 	graph $G$
    has a cyclic interval coloring is 
		$\mathcal{NP}$-complete \cite{KubaleNadolski}
    and there are concrete examples of connected graphs having
    no cyclic interval coloring 
		\cite{Nadolski,PetrosyanMkhit,AsratianCasselgrenPetrosyan}.
    Trivially, any  multigraph with
    an interval coloring also has a cyclic interval coloring
    with $\Delta(G)$ colors, but the converse does not hold \cite{Nadolski}.
    Graphs that have been proved to admit cyclic interval colorings
    (but not interval colorings) include all complete 
		multipartite graphs \cite{AsratianCasselgrenPetrosyan},
    all Eulerian bipartite graphs of maximum degree at most $8$
    \cite{AsratianCasselgrenPetrosyan}, and some 
		families of $(a,b)$-biregular graphs
    \cite{AsratianCasselgrenPetrosyan, 
		CasselgrenPetrosyanToft, CasselgrenToft}.

    In this paper we study upper and lower bounds on the number
    of colors in interval and cyclic interval colorings of various
    families of graphs. For a multigraph 
		$G\in \mathfrak{N}$ ($G\in \mathfrak{N}_{c}$), 
		we denote by $W(G)$ ($W_{c}(G)$) and $w(G)$ 
		($w_{c}(G)$) the maximum and the minimum
		number of colors in a (cyclic) interval coloring 
		of a multigraph $G$,
    respectively. 
		There are some previous results on these parameters in the literature. 
		In particular, Asratian and Kamalian 
		\cite{AsratianKamalian,AsratianKamalian2}
    proved the fundamental result that if $G \in \mathfrak{N}$ is a 
		triangle-free graph, 
		then $W(G)\leq \vert V(G)\vert -1$; this upper bound is
		sharp for e.g. complete bipartite graphs. Kamalian
		\cite{KamalianThesis} proved that if $G\in \mathfrak{N}$ 
		has at least two vertices, then 
		$W(G)\leq 2 \vert V(G)\vert-3$. 
		In \cite{GiaroKubaleMala}, it was noted that this
		upper bound can be slightly improved if the graph 
		has at least three vertices. 
		Petrosyan \cite{Petrosyan} proved that these upper bounds are
		asymptotically sharp by showing that for any $\epsilon >0$, there is
		a connected interval colorable graph $G$ 
		satisfying $W(G)\geq
		(2-\epsilon)\vert V(G)\vert$.
		Kamalian \cite{KamalianPreprint,KamalianThesis} 
		proved that the complete bipartite graph $K_{a,b}$ has an interval
		$t$-coloring if and only if $a+b-\gcd(a,b)\leq t\leq a+b-1$, where
		$\gcd(a,b)$ is the greatest common divisor of $a$ and $b$, and
		Petrosyan et al. \cite{Petrosyan,PetrosyanKhachatrianTananyan} 
		showed that the $n$-dimensional
		hypercube $Q_{n}$ has an interval $t$-coloring if and only if 
		$n\leq t\leq \frac{n\left(n+1\right)}{2}$.\\

   For cyclic interval colorings, Petrosyan and Mkhitaryan 
	\cite{PetrosyanMkhit}
    suggested the following:

\begin{conjecture}
\label{conj:upperbound}
    \begin{itemize}
        \item[(i)] For any triangle-free graph $G \in \mathfrak{N}_c$,
        $W_{c}(G)\leq \vert V(G)\vert$.

        \item[(ii)] For any graph $G \in \mathfrak{N}_c$ with 
				at least two vertices, $W_{c}(G)\leq 2\vert V(G)\vert-3$.
        \end{itemize}
    \end{conjecture}

		If true, then these upper bounds are sharp \cite{PetrosyanMkhit}.
    Using results on upper bounds for $W(G)$
    \cite{AsratianKamalian, KamalianThesis,GiaroKubaleMala}, Petrosyan and
    Mkhitaryan \cite{PetrosyanMkhit} proved
    that 
		\begin{itemize}
		
		\item $W_{c}(G)\leq \vert V(G)\vert +\Delta(G)-2$ 
		if  $G\in \mathfrak{N}_{c}$ is a triangle-free graph,
		and 
		
		\item $W_{c}(G)\leq 2\vert V(G)\vert + \Delta(G)-5$
    for any graph  $G\in \mathfrak{N}_{c}$ with at least three vertices.
		
		\end{itemize}

    In this paper we present some new general upper and lower bounds
    on the number of colors in interval and cyclic interval colorings of graphs.

    For cyclic interval colorings, we prove several results related 
		to Conjecture \ref{conj:upperbound}.
    In particular, we give improvements of the above-mentioned
		general upper 
		bounds on $W_{c}(G)$ by Petrosyan and Mkhitaryan,
		%in \cite{PetrosyanMkhit}
		and we show that 
		approximate versions of the conjecture 
		hold for several families of graphs.
    We also prove the new general upper bound
		that for any triangle-free graph $G \in \mathfrak{N}_c$,
		$W_{c}(G)\leq \frac{\sqrt{3}+1}{2}(\vert V(G)\vert-1)$, which
		is an improvement of the bound proved by Petrosyan et al.
		\cite{PetrosyanMkhit} for graphs with large maximum degree.
		These results are proved in Section 3, where we also
		prove that Conjecture \ref{conj:upperbound} (i) is true for graphs with 
		maximum degree at most $4$.

    For interval colorings, we prove that
    if $G$ is a $2$-connected multigraph and $G\in \mathfrak{N}$, 
		then
    $W(G)\leq 1+\left\lfloor \frac{|V(G)|}{2}\right\rfloor(\Delta(G)-1)$.
    This result is proved in Section 4, where we also show 
		that this upper bound is sharp, and give some related results. 
		In Section 5, we obtain some lower bounds on
    $w(G)$ for multigraphs $G \in \mathfrak{N}$.
		We show that if $G$ is an interval colorable multigraph, then
    $w(G) \geq \left\lceil\frac{|V(G)|}
		{2\alpha^{\prime}(G)}\right\rceil\delta(G)$, where
    $\alpha^{\prime}(G)$ is the size of a maximum matching
		in $G$. In particular, this implies that if $G$ has no perfect matching
    and $G\in \mathfrak{N}$, then $w(G)\geq \max\{\Delta(G),2\delta(G)\}$. Additionally, we prove that
    the same conclusion holds under the assumption that
    all vertex degrees in $G$ are odd, $G\in \mathfrak{N}$ and $\vert E(G)\vert - \frac{\vert V(G)\vert}{2}$
    is odd. All these lower bounds are sharp.
\bigskip

    %The rest of the paper is organized as follows.
    %In Section 2 we prove our lower bounds on the number of
    %colors in interval colorings. Section 3 contains
    %our proofs on upper bounds on the number of colors in interval colorings.
    %In Section 4 we prove some upper bounds on the number of colors
    %in cyclic interval colorings.

\subsection{Notation and preliminary results}\

In this section we introduce some terminology and notation,
and state some auxiliary results.

 The set of neighbors of a vertex $v$ in $G$
is denoted by $N_{G}(v)$. The degree of a vertex $v\in V(G)$ is
denoted by $d_{G}(v)$ (or just $d(v)$), the maximum degree of
vertices in $G$ by $\Delta(G)$, the minimum degree of
vertices in $G$ by $\delta(G)$, and the average degree of $G$ by
$d(G)$.

A multigraph $G$ is \emph{even} (\emph{odd}) if the degree of every
vertex of $G$ is even (odd). $G$ is \emph{Eulerian} if
it has a closed trail containing every edge of $G$. For two distinct
vertices $u$ and $v$ of a multigraph $G$, let $E(uv)$ denote the set
of all edges of $G$ joining $u$ with $v$.

The diameter of $G$, i.e. the greatest distance between any pair of
vertices in $G$, is denoted by $\mathrm{diam}(G)$, and the
circumference of $G$, i.e. the length of a longest cycle in $G$, is
denoted by $c(G)$. We denote by $\alpha^{\prime}(G)$ the size of a maximum matching in $G$, and by $\chi^{\prime}(G)$ the
chromatic index of $G$. 

If $\alpha $ is a proper edge coloring of $G$ and $v\in V(G)$, then
$S_{G}\left(v,\alpha \right)$ (or $S\left(v,\alpha \right)$) denotes
the set of colors appearing on edges incident to $v$. The smallest
and largest colors of $S\left(v,\alpha \right)$ are denoted by
$\underline S\left(v,\alpha \right)$ and $\overline S\left(v,\alpha
\right)$, respectively.\\

We shall use the following theorem due to Asratian and Kamalian
\cite{AsratianKamalian,AsratianKamalian2}.

\begin{theorem}
\label{th:AsratianKamalian} If $G$ is a triangle-free graph and
$G\in \mathfrak{N}$, then
\begin{center}
$W(G)\leq \vert V(G)\vert -1$.
\end{center}
\end{theorem}

In particular, from this result it follows that if $G$ is a
%connected 
bipartite graph and $G\in \mathfrak{N}$, then 
$W(G)\leq \vert V(G)\vert -1$.\\

For general graphs, Kamalian \cite{KamalianThesis} proved the following.

\begin{theorem}
\label{th:Kamalian} If $G$ is a graph with at least two
vertices and $G\in \mathfrak{N}$, then
\begin{center}
$W(G)\leq 2\vert V(G)\vert -3$.
\end{center}
\end{theorem}

Note that the upper bound in Theorem \ref{th:Kamalian} is sharp for
$K_{2}$, but if $G\neq K_{2}$, then this upper bound can be
improved.

\begin{theorem}
\label{th:GiaroKubale} \cite{GiaroKubaleMala}  If $G$ is a graph
with at least three vertices and $G\in \mathfrak{N}$, then
\begin{center}
$W(G)\leq 2\vert V(G)\vert -4$.
\end{center}
\end{theorem}

We shall also use the following result due to Asratian and Kamalian
\cite{AsratianKamalian,AsratianKamalian2}.

\begin{proposition}
\label{proposition} If $G\in \mathfrak{N}$, then
$\chi^{\prime}(G)=\Delta(G)$. Moreover, if $G$ is a regular
multigraph, then $G\in \mathfrak{N}$ if and only if
$\chi^{\prime}(G)=\Delta(G)$.
\end{proposition}
\bigskip

\section{Cyclic interval colorings}\

In this section we prove several results related to Conjecture
\ref{conj:upperbound}. Petrosyan and Mkhitaryan
\cite{PetrosyanMkhit} proved the following general bounds.

\begin{theorem}
\label{th:PetrosyanMkhit}
\begin{itemize}
        \item[(i)] If $G$ is a connected triangle-free graph with at
least two vertices and $G\in \mathfrak{N}_{c}$, then $W_{c}(G)\leq
\vert V(G)\vert +\Delta(G)-2$.

        \item[(ii)] If $G$ is a connected graph with at least three
vertices and $G\in \mathfrak{N}_{c}$, then $W_{c}(G)\leq 2\vert
V(G)\vert +\Delta(G)-5$.
        \end{itemize}
        \end{theorem}

Part (i) of the above theorem implies that Conjecture
\ref{conj:upperbound} (i) holds for any graph with maximum degree at
most $2$. We give a slight improvement for graphs with maximum
degree at least $3$ of their general bound.

\begin{theorem}
\label{th:degree3}
        For any triangle-free graph $G \in \mathfrak{N}_c$
        with $\Delta(G) \geq 3$,
        $W_c(G) \leq |V(G)| + \Delta(G) -3$.
\end{theorem}

\begin{proof}
    Suppose, for a contradiction,
    that the theorem is false and let $G \in \mathfrak{N}$ be graph
    with $\Delta(G) \geq 3$ which does not
		satisfy the conclusion of the theorem. 
		Let $\varphi$ be a cyclic interval
    $t$-coloring of $G$ satisfying that $t > |V(G)| + \Delta(G) -3$.
    If $t$ and $1$ are considered to be consecutive colors,
    then the first color of a set of consecutive colors can
    be chosen from the set $\{1,\dots,t\}$ in precisely $t$ ways.
    Let $$C_{\varphi} = \{c : c
    \text{ is the first color of some set $S(v,\varphi)$
        for a vertex $v$ of $G$}\}.$$
    Since $t > |V(G)|$,
    there is some color $i$ satisfying that $i \notin C_{\varphi}$.
    Moreover, by rotating the colors in $\varphi$ modulo $t$, we can construct
    a cyclic interval coloring $\varphi'$ such that
    $t \notin C_{\varphi'}$. Now, by removing all edges from $G$ with colors
    $1, \dots, \Delta(G)-2$ under $\varphi'$, we obtain a graph $H$
    such that the restriction of $\varphi'$ to $H$ is an interval
    $(t-(\Delta(G)-2))$-coloring.
    Thus, by Theorem \ref{th:AsratianKamalian}, we have that
    $t-(\Delta(G)-2)\leq |V(G)| -1$, which is a contradiction and
    the result follows.
\end{proof}

Part (ii) of Theorem \ref{th:PetrosyanMkhit} implies
that Conjecture \ref{conj:upperbound} (ii) holds
for graphs with maximum degree $2$. Proceeding as in the proof of
Theorem \ref{th:degree3}, and using Theorem \ref{th:GiaroKubale}
instead of Theorem \ref{th:AsratianKamalian},
it is straightforward to prove the following improvement of that result.

\begin{theorem}
\label{th:degree3general}
        For any graph $G \in \mathfrak{N}_c$
        with $\Delta(G) \geq 3$,
        $W_c(G) \leq 2|V(G)| + \Delta(G) -6$.
\end{theorem}

    Theorems \ref{th:degree3} and \ref{th:degree3general} imply that
    Conjecture
    \ref{conj:upperbound}  holds
    for any graph with maximum degree at most
    $3$. Next, we prove that Conjecture \ref{conj:upperbound} (i)
    holds for any graph with maximum degree at most $4$.

\begin{theorem}
\label{thm:maxdegree4}
    For any triangle-free graph $G\in \mathfrak{N}_c$ with $\Delta(G) \leq 4$, $W_c(G) \leq |V(G)|$.
\end{theorem}

\begin{proof}
    Suppose that the theorem is false and let $G$ be a
    vertex-minimal counterexample with maximum degree $4$.
    Let $\varphi$ be a cyclic interval $t$-coloring of $G$,
    where $t > |V(G)|$.
    %Moreover, we assume that among the vertex-minimal counterexamples
    %to the proposition, $G$ is a counterexample with minimal number
    %of edges.

    Since $t > |V(G)|$, $\varphi$ is not an interval coloring.
    Moreover, by a counting argument similar to the one in the proof
    of the preceding theorems we may assume that at most
    two vertices in $G$ are {\em cyclic} vertices under $\varphi$,
    that is, they are incident to one edge colored $1$
    and one edge colored $t$.

    \bigskip

    Suppose first that only one vertex $v$
    in $G$ is cyclic under $\varphi$. From $G$ we define
    a new graph $H$ by splitting $v$ into two new vertices
    $v'$ and $v''$
    where
    \begin{itemize}

    \item $v'$ is adjacent to all vertices $x$ such that
    $xv \in E(G)$ and $\varphi(xv) \in \{t-2,t-1,t\}$;

    \item $v''$ is adjacent to all vertices $x$ such that
    $xv \in E(G)$ and $\varphi(xv) \in \{1,2,3\}$.

    \end{itemize}
		The edge coloring $\varphi$ induces an edge coloring
    $\varphi_H$ of $H$, and since $v$ is the only cyclic vertex
    of $G$, $\varphi_H$ is an interval $t$-coloring. Hence,
		by Theorem \ref{th:AsratianKamalian},
    $t \leq |V(H)| -1 = |V(G)|$, a contradiction.

\bigskip

    Suppose now instead that two vertices $u, v \in V(G)$ are cyclic.
    We first consider the case when $u$ and $v$ are adjacent.
    We shall assume that $\varphi(uv) \in \{1,2,3\}$
    (the case when $\varphi(uv) \in \{t-2, t-1, t\}$ can be done
    analogously).

    From $G$ we define a new graph $H$ by splitting the vertex
    $v$ into two new vertices $v'$ and $v''$, and the vertex
    $u$ into two new vertices $u'$ and $u''$
    where
    \begin{itemize}

    \item $v'$ is adjacent to all vertices $x\neq u$ such that
    $xv \in E(G)$ and $\varphi(xv) \in \{t-2,t-1,t\}$;

    \item $v''$ is adjacent to all vertices $x \neq u$ such that
    $xv \in E(G)$ and $\varphi(xv) \in \{1,2,3\}$;

    \item $u'$ is adjacent to all vertices $x \neq v$ such that
    $xu \in E(G)$ and $\varphi(xu) \in \{t-2,t-1,t\}$;

    \item $u''$ is adjacent to all vertices $x \neq v$ such that
    $xu \in E(G)$ and $\varphi(xu) \in \{1,2,3\}$.

    \end{itemize}
    Let $\varphi_H$ be the edge coloring of $H$ induced by
    $\varphi$. Let $J = H + \{u'v', u''v''\}$ and
    color the edge $u'v'$ with color $t+1$ and $u''v''$
    with color $\varphi(uv)$. Denote the obtained coloring
    of $J$ by $\varphi_J$.
    Since $u$ and $v$ are the only cyclic vertices of $G$,
    $\varphi_J$ is an interval $(t+1)$-coloring. Consequently,
    $t +1 \leq |V(J)| -1 = |V(G)| +1$, which contradicts our assumption,
    and the result follows.

\bigskip

    Let us now consider the case when $u$ and $v$ are not adjacent.
    We first prove that we may assume that
    $S_G(v,\varphi) = S_G(u,\varphi)   = \{t-1,t,1,2\}$.
    Suppose, for example, that $t-1$ does not appear at $u$.
    Let $e$ be the edge incident with $u$ colored $t$ and set $G' = G-e$.
    Then the restriction of $\varphi$ to $G'$ is a cyclic interval coloring
    with $t$ colors and $|V(G')| = |V(G)|$. Moreover, in the coloring of
    $G'$ there is a only one cyclic vertex, and we may thus proceed as
    above. A similar argument applies if color $2$
    does not appear at one of the vertices
    $u,v$. We thus conclude that
    we may assume that
    $S_G(v,\varphi) = S_G(u,\varphi) =  \{t-1,t,1,2\}$.

    As above, we form a new graph $H$ from $G$ by splitting the vertex
    $u$ into two new vertices $u'$ and $u''$,
    and $v$ into two vertices $v'$ and $v''$,
    where
    \begin{itemize}

    \item $u'$ ($v'$) is adjacent to every neighbor of $u$ ($v$) in
    $G$ that is
    joined to $u$ ($v$) by an edge with color in $\{t-2,t-1,t\}$;

    \item $u''$ ($v''$) is adjacent to every neighbor
    of $u$ ($v$) in $G$ that is
    joined to $u$ ($v$) by an edge with color in $\{1,2,3\}$.

    \end{itemize}

    Let $\varphi_H$ be the interval $t$-coloring of $H$ induced by $\varphi$.
    Moreover, let $x'$ and $y'$ be the neighbors of $u'$ and $v'$,
    respectively, that are colored $t$ under $\varphi_H$, and
    $x''$ and $y''$  be the neighbors of $u''$ and $v''$, respectively,
    that are colored $1$ under $\varphi_H$. Note that all these vertices
    are distinct, since $u$ and $v$ are the only cyclic vertices
    in $G$.

    \bigskip

    We first consider the case when $x'y' \in E(H)$ or $x''y'' \in E(H)$.
    Suppose e.g. that $x'y' \in E(H)$. Then, since $H$ is triangle-free and
    only colors $t-1$ and $t$ appear at $u'$ and $v'$,
    %$u'y', v'x' \notin E(H)$. Moreover,
    $u'$ and $v'$ has no common neighbor in $H$. Hence, the graph $H + u'v'$
    is triangle-free and by coloring $u'v'$ with color $t+1$, we obtain
    an interval $(t+1)$-coloring of $H + u'v'$.
    Thus $t+1 \leq |V(H)|-1 = |V(G)|+1$,
    a contradiction. The case when $x''y'' \in E(H)$ is analogous.

    \bigskip

    Suppose now that $x'y', x''y'' \notin E(H)$. If $u'$ and $v'$
    have no common neighbor in $H$, then we proceed as in the
    preceding paragraph, so assume that $u'$ and $v'$ have a common neighbor.
    This means that either $x' \in N_H(u') \cap N_H(v')$ or
    $y' \in N_H(u') \cap N_H(v')$. Suppose, for instance, that the former holds.
    %Then $v'$ and $x'$ have no common neighbor. 
		We first consider the case
    when $u'y' \notin E(H)$. 
		
		If $N_H(x') \cap N_H(y') = \{v'\}$, then neither $x' v'$
		or $v' y'$
		is contained in a cycle of length at most $4$, 
		because $d_H(v')=2$.
		Thus the graph $J$ obtained from $H- v'y'$ by identifying
		$v'$ and $y'$ is triangle-free. Moreover, the coloring $\varphi_H$
		induces an interval $t$-coloring of $J$, and, consequently,
		$t \leq |V(J)|-1 = |V(G)|$, which is a contradiction.
		
		If $| N_H(x') \cap N_H(y')|  = 2$, 
		then we proceed as follows. If there is an edge colored $t$
		under $\varphi_H$
		that is not incident with $u'$ or $v'$, then we form 
		a graph $J$ by removing the edge $u'x'$ from $H-v'$.
		Note that the restrction of $\varphi_H$  to $J$ is an
		interval $t$-coloring of $J$. Hence $t \leq |V(J)|-1 = |V(G)|$,
		a contradiction, as above. If all edges colored
		$t$ under $\varphi_H$ are incident with $u'$ or $v'$, 
		then the restriction of $\varphi_H$ to the graph $H-\{u',v'\}$
		is an interval $(t-1)$-coloring, unless $d_H(y')=1$. 
		If $d_H(y')> 1$, then it follows that $t-1 \leq |V(H)|-3$,
		again a contradiction to our above assumption.
		If $d_H(y') =1$, then a similar contradiction can be derived
		by considering the graph $H-\{u',v', y'\}$ and the restriction
		of $\varphi_H$ to this graph.
	
    Now we consider the case when $u'y' \in E(H)$.
    Suppose first that there is some vertex
    $w \in V(H) \setminus \{u',v', x', y'\}$ where color $t-1$ appears
    under $\varphi_H$. Then the restriction of $\varphi_H$ to the graph
    $H-\{u',v'\}$ is an interval $t'$-coloring, for some $t' \geq t-1$,
    and, consequently,
    $t -1 \leq |V(H)|- 2 -1 = |V(G)|-1$, a contradiction.

    Now assume that there is no vertex
    $w \in V(H) \setminus \{u',v', x', y'\}$ where color $t-1$ appears
    under $\varphi_H$. Then $u$ and $v$ are the only vertices in $G$
    where colors $t-1$ and $t$ appear under $\varphi$.
    Denote by $x$ and $y$ respectively the vertices in $G$ such that
    $\varphi(ux) =t$ and $\varphi(vy) =t$.
    Consider the graph $J$ obtained from $G -\{ux, uy, vx, vy\}$
    by identifying $u$ and $x$,
    and $v$ and $y$, respectively.
    $J$ is triangle-free and the coloring $\varphi$
    induces an cyclic interval $(t-2)$-coloring of $J$.
    Hence, $J$ is a counterexample with a smaller number of vertices
    than $G$, a contradiction to the choice of $G$.
\end{proof}

    In the following we shall derive some upper bounds on $W_c(G)$ 
		related
    to Conjecture \ref{conj:upperbound}. To this end, we 
		first introduce some new
    notation. 
		For brevity, a graph $G$ is called an 
		\emph{$(n,m)$-graph} if it contains $n$
		vertices and $m$ edges.
		Let $G$ be an $(n,m)$-graph
    and $G\in \mathfrak{N}_{c}$. Also, let $\alpha$ be a cyclic interval
    $W_{c}(G)$-coloring of $G$ and $U$ be the set of all vertices
    $v\in V(G)$ such that
    $S(v,\alpha)$ is not an interval of integers. If $U=\emptyset$, then
    $G\in \mathfrak{N}$ and 
		{\black $W_{c}(G)\leq W(G)\leq 2|V(G)|-3$},
    by Theorem \ref{th:Kamalian}. Assume that $U\neq \emptyset$.
    Clearly, for each
    $u\in U$, there exists a color $c_{u}$ ($1<c_{u}<W_{c}(G)$) such
    that $c_{u}\notin S(u,\alpha)$. For each $u\in U$, we split the
    neighbors of $u$ into two disjoint sets as follows:
    $N_{G}(u)=N_{\alpha}^{<}(u)\cup N_{\alpha}^{>}(u)$, where
    $N_{\alpha}^{<}(u)=\{v\in N_{G}(u)\colon\,1\leq \alpha(uv)<c_{u}\}$
    and $N_{\alpha}^{>}(u)=\{v\in N_{G}(u)\colon\,c_{u}< \alpha(uv)\leq
    W_{c}(G)\}$. (Note that these sets are uniquely determined
		from the cyclic interval coloring $\alpha$.)
		We construct a new graph $S_{\alpha,U}(G)$ by splitting
    all the vertices of $U$ as follows: for each $u\in U$, we delete $u$
    from $G$ and add two new vertices $u^{<}$ and $u^{>}$; then we join
    the vertex $u^{<}$ with the vertices $N_{\alpha}^{<}(u)$ and color
    the edge $u^{<}v$ by the color $\alpha(uv)$ for every $v\in
    N_{\alpha}^{<}(u)$; next we join the vertex $u^{>}$ with the
    vertices $N_{\alpha}^{>}(u)$ and color the edge $u^{>}v$ by the
    color $\alpha(uv)$ for every $v\in N_{\alpha}^{>}(u)$; finally we
    color all the remaining edges using the same colors as in $G$.
    Clearly, $S_{\alpha,U}(G)$ is an $(n+\vert U\vert,m)$-graph and
    $S_{\alpha,U}(G)\in \mathfrak{N}$.

\begin{lemma}
\label{ourlemma}(Translation Lemma) 
Let $\mathfrak{C}\subseteq
\mathfrak{N}_{c}$, $\mathfrak{C^{\prime}}\subseteq \mathfrak{N}$ and
assume that $W(G^{\prime})\leq f(\vert V(G^{\prime})\vert)$ holds
for any graph $G' \in \mathfrak{C^{\prime}}$, where $f$ is a
monotonically non-decreasing function. If for every $G\in
\mathfrak{C}$, there is a cyclic interval $W_{c}(G)$-coloring
$\alpha$ of $G$ and a subset $U\subseteq V(G)$, such that the graph
$S_{\alpha,U}(G)$ belongs to $\mathfrak{C^{\prime}}$, then for any
$G\in \mathfrak{C}$, we have
\begin{center}
$W_{c}(G)\leq f\left(\vert V(G)\vert+\left\lfloor\frac{2\vert
E(G)\vert-\vert V(G)\vert}{W_{c}(G)}\right\rfloor\right)$.
\end{center}
\end{lemma}

\begin{proof}
Let $G$ be an $(n,m)$-graph and $G\in \mathfrak{C}$. Also, let
$t=W_{c}(G)$. Consider a cyclic interval $t$-coloring $\alpha$ of
$G$. We say the color $i$ ($1\leq i\leq t$) \textit{splits} 
the vertex $v\in
V(G)$ if $\{i-1,i\}\subseteq S(v,\alpha)$ (modulo $t$). For each
$v\in V(G)$, let $S^{\prime}(v,\alpha)$ denote the set of all colors
$i$ that splits the vertex $v$. Clearly, $\vert
S^{\prime}(v,\alpha)\vert=d_{G}(v)-1$ for every $v\in V(G)$. By the
pigeonhole principle, there exists a color that splits no more than
$\left\lfloor \frac{{\sum_{v\in
V(G)}\left(d_{G}(v)-1\right)}}{t}\right\rfloor=\left\lfloor
\frac{2m-n}{t}\right\rfloor$ vertices. Without loss of generality we
may assume that this color is $1$ (by rotating the colors in
$\alpha$ modulo $t$, we can construct such a coloring). 
Let $U$ be the
set of vertices of $G$ that are splitted by color $1$. Then
$\vert U\vert \leq \left\lfloor \frac{2m-n}{t}\right\rfloor$. 
Let us now consider the graph $S_{\alpha,U}(G)$ defined above. 
If $S_{\alpha,U}(G)\in
\mathfrak{C^{\prime}}$, then
\begin{center}
$W_{c}(G)=t\leq W\left(S_{\alpha,U}(G)\right)\leq f(\left\vert
V\left(S_{\alpha,U}(G)\right)\right\vert)\leq f\left(n+\left\lfloor
\frac{2m-n}{W_{c}(G)}\right\rfloor\right)$.
\end{center} %~$\square$
\end{proof}

The following two theorems are improvements of the bounds
on $W_c(G)$ in Theorem \ref{th:PetrosyanMkhit} for graphs
with large maximum degree.

\begin{theorem}
\label{mytheorem1} If $G$ is a triangle-free graph and $G\in
\mathfrak{N}_{c}$, then
\begin{center}
$W_{c}(G)\leq \frac{\sqrt{3}+1}{2}(\vert V(G)\vert-1)$.
\end{center}
\end{theorem}
\begin{proof}
Let $G$ be a triangle-free $(n,m)$-graph and $\mathfrak{C}$
($\mathfrak{C^{\prime}}$) be the set of all cyclically interval
colorable (interval colorable) triangle-free graphs. 
By Theorem \ref{th:AsratianKamalian}, we have
$W(G^{\prime})\leq \vert V(G^{\prime})\vert-1$ for every
$G^{\prime}\in \mathfrak{C^{\prime}}$.

Let
$t=W_{c}(G)$ and consider a cyclic interval $t$-coloring $\alpha$ of
$G$.  As in the proof of Lemma
\ref{ourlemma}, we construct a graph $S_{\alpha,U}(G)$ from
the coloring $\alpha$, $G$ and $U$, where 
$\vert U\vert \leq \left\lfloor\frac{2m-n}{t}\right\rfloor$.
Since $S_{\alpha,U}(G)\in
\mathfrak{C^{\prime}}$, by Lemma \ref{ourlemma}, we obtain

\begin{center}
$t\leq W\left(S_{\alpha,U}(G)\right)\leq \left\vert
V\left(S_{\alpha,U}(G)\right)\right\vert -1 \leq n+\left\lfloor
\frac{2m-n}{t}\right\rfloor-1$.
\end{center}

On the other hand, since $G$ is triangle-free, by Mantel's theorem,
we have $2m\leq \frac{n^{2}}{2}$. Taking this into account, we obtain

\begin{center}
$t\leq n+\left\lfloor \frac{2m-n}{t}\right\rfloor-1\leq
n+\frac{\frac{n^{2}}{2}-n}{t}-1$.
\end{center}

This yields a quadratic inequality
$t^{2}-(n-1)t-n(\frac{n}{2}-1)\leq 0$. Thus,

\begin{center}
$W_{c}(G)=t\leq \frac{1}{2}\left(n-1+\sqrt{3n^{2}-6n+1}\right)\leq
\frac{\sqrt{3}+1}{2}(\vert V(G)\vert-1)$.
\end{center}
%~$\square$
\end{proof}

\begin{theorem}
\label{mytheorem2} If $G$ is a graph with at least two
vertices and $G\in \mathfrak{N}_{c}$, then
\begin{center}
$W_{c}(G)\leq (\sqrt{3}+1)\vert V(G)\vert-3$.
\end{center}
\end{theorem}
\begin{proof}
Let $G$ be a $(n,m)$-graph ($n\geq 2$) and $\mathfrak{C}$
($\mathfrak{C^{\prime}}$) be the set of all cyclically interval
colorable (interval colorable) graphs with at least two
vertices. By Theorem \ref{th:Kamalian}, we have 
$W(G^{\prime})\leq 2\vert V(G^{\prime})\vert
-3$ for every $G^{\prime}\in \mathfrak{C^{\prime}}$.

Set $t=W_{c}(G)$ and consider a cyclic interval
$t$-coloring $\alpha$ of $G$. As in the proof
of Lemma \ref{ourlemma},
we construct a graph $S_{\alpha,U}(G)$ from
the coloring $\alpha$, $G$ and $U$, where 
$\vert U\vert \leq \left\lfloor\frac{2m-n}{t}\right\rfloor$.
Since $S_{\alpha,U}(G)$ belongs to
$\mathfrak{C^{\prime}}$, we deduce from Lemma \ref{ourlemma} that
\begin{center}
$t\leq W\left(S_{\alpha,U}(G)\right)\leq 2\left\vert
V\left(S_{\alpha,U}(G)\right)\right\vert -3
\leq 2\left(n+\left\lfloor \frac{2m-n}{t}\right\rfloor\right)-3$.
\end{center}

On the other hand, since $2m\leq n(n-1)$, we obtain

\begin{center}
$t\leq 2\left(n+\left\lfloor
\frac{2m-n}{t}\right\rfloor\right)-3\leq
2\left(n+\frac{n(n-1)-n}{t}\right)-3$.
\end{center}

This yields a quadratic inequality $t^{2}-(2n-3)t-2n(n-2)\leq 0$.
Thus,

\begin{center}
$W_{c}(G)=t\leq
\frac{1}{2}\left(2n-3+\sqrt{12n^{2}-28n+9}\right)\leq
(\sqrt{3}+1)\vert V(G)\vert-3$.
\end{center} %~$\square$
\end{proof}

Our next result is an improvement of part (ii) of Theorem
\ref{th:PetrosyanMkhit} for graphs $G$ with average degree $d(G) <\Delta(G)$.

\begin{theorem}
\label{translation-average-degree} If $G$ is a graph with
at least three vertices and $G\in \mathfrak{N}_{c}$, then
\begin{center}
$W_{c}(G)\leq 2\vert V(G)\vert + d(G) - 3.5$.
\end{center}
\end{theorem}

\begin{proof}
Let $G$ be an $(n,m)$-graph ($n\geq 3$) and $\mathfrak{C}$
($\mathfrak{C^{\prime}}$) be the set of all cyclically interval
colorable (interval colorable) graphs with at least two
vertices. Also, let $t=W_{c}(G)$ and $d=d(G)$. 
By proceeding as above, 
we consider a cyclic
interval $t$-coloring $\alpha$ of $G$
and deduce, using Lemma \ref{ourlemma}, that
\begin{center}
$t\leq W\left(S_{\alpha,U}(G)\right)\leq 2\left\vert
V\left(S_{\alpha,U}(G)\right)\right\vert -3
\leq 2\left(n+\left\lfloor \frac{2m-n}{t}\right\rfloor\right)-3$.
\end{center}

On the other hand, since $2m\leq dn$, we obtain

\begin{center}
$t\leq 2\left(n+\left\lfloor
\frac{2m-n}{t}\right\rfloor\right)-3\leq
2\left(n+\frac{dn-n}{t}\right)-3$.
\end{center}

As above, this yields a quadratic inequality $t^{2}-(2n-3)t-2n(d-1)\leq 0$.
Thus,

\begin{center}
$W_{c}(G)=t\leq
\frac{1}{2}\left(2n-3+\sqrt{4n^{2}-(20-8d)n+9}\right)\leq 2\vert
V(G)\vert+d(G)-3.5$.
\end{center} %~$\square$
\end{proof}

Since any planar graph has average degree less than $6$, we deduce the following from the preceding theorem.

\begin{corollary}
\label{translation-planar} If $G$ is a planar graph and
$G\in \mathfrak{N}_{c}$, then
\begin{center}
$W_{c}(G)\leq 2\vert V(G)\vert+2$.
\end{center}
\end{corollary}

For triangle-free sparse graphs $G$ with a cyclic interval coloring we can prove a better upper bound on $W_c(G)$.

\begin{theorem}
\label{mytheorem3} 
Let $G\in \mathfrak{N}_{c}$ be a triangle-free $(n,m)$-graph, and
$a$ and $b$ positive numbers satisfying
$m\leq a\cdot n+b$. Then if $8b+1\leq
(3-4a)^{2}$, then
\begin{center}
$W_{c}(G)\leq \vert V(G)\vert+2a-2$.
\end{center}
\end{theorem}
\begin{proof}
Let $G$ be a triangle-free $(n,m)$-graph ($m\leq a\cdot n+b$) and
$\mathfrak{C}$ ($\mathfrak{C^{\prime}}$) be the set of all
cyclically interval colorable (interval colorable) triangle-free
graphs. Also, let $t=W_{c}(G)$. Consider a cyclic interval
$t$-coloring $\alpha$ of $G$. 
As in the proof of Theorem \ref{mytheorem1}, we have that

\begin{center}
$t\leq n+\left\lfloor
\frac{2m-n}{t}\right\rfloor-1$.
\end{center}

On the other hand, since $m\leq a\cdot n+b$, we obtain

\begin{center}
$t\leq n+\left\lfloor \frac{2m-n}{t}\right\rfloor-1\leq
n+\frac{2an+2b-n}{t}-1$.
\end{center}
This yields a quadratic inequality
$t^{2}-(n-1)t-\left(n(2a-1)+2b\right)\leq 0$, implying that

\begin{center}
$W_{c}(G)=t\leq
\frac{1}{2}\left(n-1+\sqrt{n^{2}-2n(3-4a)+8b+1}\right)$.
\end{center}

Now taking into account that $8b+1\leq (3-4a)^{2}$, we obtain

\begin{center}
$W_{c}(G)\leq
\frac{1}{2}\left(n-1+\sqrt{n^{2}-2n(3-4a)+8b+1}\right)\leq \vert
V(G)\vert+2a-2$.
\end{center} 
\end{proof}

\begin{corollary}
\label{mycorollary2.1} If $G$ is a triangle-free planar graph and
$G\in \mathfrak{N}_{c}$, then
\begin{center}
$W_{c}(G)\leq \vert V(G)\vert+2$.
\end{center}
\end{corollary}

\begin{corollary}
\label{mycorollary2.2} If $G$ is a triangle-free outerplanar graph
with at least two vertices and $G\in \mathfrak{N}_{c}$, then
\begin{center}
$W_{c}(G)\leq \vert V(G)\vert+1$.
\end{center}
\end{corollary}

We note that the last two results are almost sharp since
$W_{c}(C_{n}) = n$, where $C_{n}$ is a cycle on $n$ vertices.
\bigskip

%%%%%%%%%%%%%%%%%%%%%%%%%%%%%%%

\section{Bounds on $W(G)$}\

In this section we prove some bounds on $W(G)$ for interval
colorable multigraphs $G$. Let us first recall
some other bounds on $W(G)$ that appear in the literature.

The first lower bound
on $W(G)$ for interval colorable regular graphs $G$ was obtained by
Kamalian \cite{KamalianThesis}. In particular, he proved that if $G$
is a regular graph with $\chi^{\prime}(G)=\Delta(G)$ and $3\leq
\vert V(G)\vert\leq 2^{\Delta(G)}+1$, then $W(G)\geq
\Delta(G)+\left\lfloor \log_{2}\left(\vert
V(G)\vert-1\right)\right\rfloor$. In \cite{Axenovich}, Axenovich
proved that if $G$ is a planar graph and $G\in \mathfrak{N}$, then
$W(G)\leq \frac{11}{6}\vert V(G)\vert$, and
conjectured that for all interval colorable planar graphs, $W(G)\leq
\frac{3}{2}\vert V(G)\vert$. In \cite{AsratianKamalian2}, Asratian
and Kamalian proved that if $G$ is connected and $G\in
\mathfrak{N}$, then 
$$W(G)\leq
\left(\mathrm{diam}(G)+1\right)\left(\Delta(G)-1\right) +1.$$ 
If we, in addition,
assume that $G$ is bipartite, then this bound can be improved to
$W(G)\leq \mathrm{diam}(G)\left(\Delta(G)-1\right) +1$. Recently,
Kamalian and Petrosyan \cite{KamalianPetrosyan} showed that these
upper bounds cannot be significantly improved. In this section we
shall derive a similar upper bound on $W(G)$ for interval colorable multigraphs $G$ based on the circumference of $G$. First we give a
short proof of Theorem \ref{th:Kamalian} based on Theorem
\ref{th:AsratianKamalian}; the original proof by Kamalian
is rather lengthy.\\

\begin{proof}[Proof of Theorem \ref{th:Kamalian}.]
Let
$V(G)=\{v_{1},v_{2},\ldots,v_{n}\}$ and $\alpha$ be an interval
$W(G)$-coloring of the graph $G$.
 Define an auxiliary graph $H$ as follows:
\begin{center}
$V(H)=U\cup W$, where
\end{center}
\begin{center}
$U=\{u_{1},u_{2},\ldots,u_{n}\}$, $W=\{w_{1},w_{2},\ldots,w_{n}\}$
and
\end{center}
\begin{center}
$E(H)=\left\{u_{i}w_{j},u_{j}w_{i}\colon\,v_{i}v_{j}\in E(G), 1\leq i\leq n,1\leq j\leq n\right\}\cup \{u_{i}w_{i}\colon\,1\leq i\leq n\}$.
\end{center}

Clearly, $H$ is a %connected 
bipartite graph with $\vert V(H)\vert = 2\vert V(G)\vert$.\\

Define an edge-coloring $\beta$ of $H$ as follows:
\begin{description}
\item[(1)] for every edge $v_{i}v_{j}\in E(G)$, let $\beta (u_{i}w_{j})=\beta(u_{j}w_{i})=\alpha(v_{i}v_{j})+1$,

\item[(2)] for $i=1,2,\ldots,n$, let $\beta(u_{i}w_{i})=\overline S(v_{i},\alpha)+2$.
\end{description}

It is easy to see that $\beta$ is an edge-coloring of the graph $H$ with colors $2,3,\ldots,W(G)+2$ and $\underline S(u_{i},\beta)=\underline
S(w_{i},\beta)$ for $i=1,2,\ldots,n$. Now we present an interval
$(W(G)+2)$-coloring of the graph $H$. For that we take one edge
$u_{i_{0}}w_{i_{0}}$ with $\underline S\left(u_{i_{0}},\beta\right)=\underline
S\left(w_{i_{0}},\beta\right)=2$, and recolor it with color $1$. Clearly, such a coloring is an interval $(W(G)+2)$-coloring of $H$. 
Since $H$ is a %connected 
bipartite graph and $H\in \mathfrak{N}$, by Theorem
\ref{th:AsratianKamalian}, we have
\begin{center}
$W(G)+2\leq \vert V(H)\vert -1 = 2\vert V(G)\vert-1$, thus
\end{center}
\begin{center}
$W(G)\leq 2\vert V(G)\vert-3$.
\end{center} %~$\square$
\end{proof}

\begin{theorem}
\label{mytheorem4} If $G$ is a $2$-connected multigraph and $G\in
\mathfrak{N}$, then
\begin{center}
$W(G)\leq 1+\left\lfloor \frac{c(G)}{2}\right\rfloor(\Delta(G)-1)$.
\end{center}
\end{theorem}
\begin{proof} Consider an interval $W(G)$-coloring
$\alpha $ of $G$. In the coloring $\alpha$ of $G$, we consider the
edges with colors $1$ and $W(G)$. Let $e=uv$,
$e^{\prime}=u^{\prime}v^{\prime}$ and $\alpha(e)=1$,
$\alpha(e^{\prime})=W(G)$. Since $G$ is $2$-connected, there is a
cycle $C$ that contains both edges $e$ and $e^{\prime}$. 
Clearly, $\vert V(C)\vert\leq c(G)$. We
label the vertices of $C$ in two directions: from $u$ to
$u^{\prime}$ and from $v$ to $v^{\prime}$. Let
$P=u_{1},\ldots,u_{s}$ and $Q=v_{1},\ldots,v_{t}$ be two paths on $C$
from
$u$ to $u^{\prime}$ and from $v$ to $v^{\prime}$, respectively,
where $u_{1}=u,u_{s}=u^{\prime}$ and $v_{1}=v,v_{t}=v^{\prime}$
($s,t\geq 1$). Clearly, $\min\{s,t\}\leq \left\lfloor
\frac{c(G)}{2}\right\rfloor$. Without loss of generality we may
assume that $s\leq \left\lfloor \frac{c(G)}{2}\right\rfloor$.

Since $\alpha$ is an interval $W(G)$-coloring of $G$, we have

\begin{center}
$\alpha(u_{1}u_{2})\leq d_{G}(u_{1})$,

$\alpha(u_{2}u_{3})\leq \alpha(u_{1}u_{2})+ d_{G}(u_{2})-1$,

$\cdots \cdots \cdots \cdots \cdots \cdots$

$\alpha(u_{i}u_{i+1})\leq \alpha(u_{i-1}u_{i})+ d_{G}(u_{i})-1$,

$\cdots \cdots \cdots \cdots \cdots \cdots$

$\alpha(u_{s-1}u_{s})\leq \alpha(u_{s-2}u_{s-1})+ d_{G}(u_{s-1})-1$,

$W(G)=\alpha(e^{\prime})=\alpha(u^{\prime}v^{\prime})\leq
\alpha(u_{s-1}u_{s})+ d_{G}(u_{s})-1$.

\end{center}

Summing up these inequalities, we obtain

\begin{center}
$W(G)\leq 1+{\sum\limits_{i=1}^{s}\left(d_{G}(u_{i})-1\right)}\leq
1+\left\lfloor\frac{c(G)}{2}\right\rfloor(\Delta(G)-1)$.
\end{center} %~$\square$
\end{proof}

\begin{corollary}
\label{mycorollary3.1} If $G$ is a $2$-connected multigraph and
$G\in \mathfrak{N}$, then
\begin{center}
$W(G)\leq 1+\left\lfloor \frac{\vert
V(G)\vert}{2}\right\rfloor(\Delta(G)-1)$.
\end{center}
\end{corollary}

\begin{corollary}
\label{mycorollary3.2} If $G$ is a $2$-connected multigraph with
$\Delta(G)\leq 4$ and $G\in \mathfrak{N}$, then
\begin{center}
$W(G)\leq 3\left\lfloor \frac{\vert V(G)\vert}{2}\right\rfloor+1$.
\end{center}
\end{corollary}

\begin{corollary}
\label{mycorollary3.3} If $G$ is a $2$-connected planar graph with
$\Delta(G)\leq 4$ and $G\in \mathfrak{N}$, then
\begin{center}
$W(G)\leq \frac{3}{2}\vert V(G)\vert$.
\end{center}
\end{corollary}
\begin{proof} First let us note that Corollary \ref{mycorollary3.2} implies that the statement is true for planar graphs with an odd number of vertices. Moreover, we have $W(G)\leq \frac{3}{2}\vert V(G)\vert+1$. Assume that $\vert V(G)\vert =2n$ ($n\in \mathbb{N}$). Consider an interval $(3n+1)$-coloring
$\alpha $ of $G$. As in the proof of Theorem \ref{mytheorem4}, we consider the edges with colors $1$ and $3n+1$. Let $e=uv$,
$e^{\prime}=u^{\prime}v^{\prime}$ and $\alpha(e)=1$,
$\alpha(e^{\prime})=3n+1$. Since $G$ is $2$-connected, there is a
cycle $C$ that contains both edges $e$ and $e^{\prime}$. 
Clearly, $\vert V(C)\vert\leq 2n$. We
label the vertices of $C$ in two directions: from $u$ to
$u^{\prime}$ and from $v$ to $v^{\prime}$. Let
$P=u_{1},\ldots,u_{s}$ and $Q=v_{1},\ldots,v_{t}$ be two paths on $C$ from
$u$ to $u^{\prime}$ and from $v$ to $v^{\prime}$, respectively,
where $u_{1}=u,u_{s}=u^{\prime}$ and $v_{1}=v,v_{t}=v^{\prime}$
($s,t\geq 1$). It is easy to see that 
$s=t=\frac{\vert V(G)\vert}{2}=n$ (otherwise, by considering the shortest path, we obtain that $W(G)\leq 3n$). Moreover, we have

\begin{center}
$\alpha(u_{1}u_{2})=\alpha(v_{1}v_{2})=4$,

$\alpha(u_{2}u_{3})=\alpha(v_{2}v_{3})=7$,

$\cdots \cdots \cdots \cdots \cdots \cdots$

$\alpha(u_{i}u_{i+1})=\alpha(v_{i}u_{i+1})=3i+1$,

$\cdots \cdots \cdots \cdots \cdots \cdots$

$\alpha(u_{s-1}u_{s})=\alpha(v_{t-1}v_{t})=3n-2$,

$\alpha(e^{\prime})=\alpha(u^{\prime}v^{\prime})=\alpha(u_{s}v_{t})=\alpha(u_{n}v_{n})=3n+1$.
\end{center}

This implies that $\overline S\left(u_{i},\alpha\right)=\overline
S\left(v_{i},\alpha\right)=3i+1$ for $i=1,\ldots,n$. Let us consider 
the edges with color $3n$. These edges can only be incident to 
vertices $u_{n}$ and $v_{n}$, since 
$\overline S\left(w,\alpha\right)<3n$ for any 
$w\in V(G)\setminus\{u_{n},v_{n}\}$. This implies that 
the edges with colors $3n$ and $3n+1$ must be parallel, which is a contradiction.
\end{proof}

The last corollary above shows that Axenovich's conjecture 
\cite{Axenovich} is true
for all interval colorable $2$-connected planar graphs with
maximum degree at most $4$.

It is well-known that if a connected regular multigraph $G$ has a
cut-vertex, then $\chi^{\prime}(G)>\Delta(G)$. Moreover, by
Proposition \ref{proposition}, if $G$ is interval colorable, then
$\chi'(G) = \Delta(G)$. Thus, if a connected regular multigraph %$G$
is interval colorable, then %$G$ 
it is $2$-connected. Hence, we have the
following two results.

\begin{corollary}
\label{mycorollary3.4} If $G$ is a connected $r$-regular multigraph
and $G\in \mathfrak{N}$, then
\begin{center}
$W(G)\leq 1+\left\lfloor \frac{\vert
V(G)\vert}{2}\right\rfloor(r-1)$.
\end{center}
\end{corollary}

\begin{corollary}
\label{mycorollary3.5} If $G$ is a connected cubic multigraph and
$G\in \mathfrak{N}$, then
\begin{center}
$W(G)\leq \vert V(G)\vert+1$.
\end{center}
\end{corollary}

We now show that all aforementioned upper bounds are sharp.

\begin{proposition}
\label{myproposition5} For any integers $n,r\geq 2$, there exists a $2$-connected multigraph $G$ with $\vert V(G)\vert =2n$ and
$\Delta(G)=r$ such that $G\in \mathfrak{N}$ and $W(G)=1+n(r-1)$.
\end{proposition}
\begin{proof} For the proof, we are going to construct a multigraph
$G_{n,r}$ that satisfies the specified conditions.
We define the multigraph $G_{n,r}$ ($n,r\geq 2$) as follows:

\begin{description}

\item[1)] $V\left(G_{n,r}\right)=\left\{u_{1},u_{2},\ldots,
u_{n},v_{1},v_{2},\ldots,v_{n}\right\}$,

\item[2)] $E\left(G_{n,r}\right)$ contains
$r-1$ parallel edges between the vertices $u_1$ and $v_1$,
and $r-1$ parallel edges between $u_n$ and $v_n$,
for $2\leq i\leq
n-1$, $u_i$ and $v_i$ are
joined by $r-2$ parallel edges, and 
for $1\leq j\leq
n-1$, $E(G_{n,r})$ contains
the edges $u_{j}u_{j+1},v_{j}v_{j+1}$.

\end{description}

$G_{n,r}$ is a $2$-connected multigraph with $\vert
V\left(G_{n,r}\right)\vert =2n$ and $\Delta\left(G_{n,r}\right)=r$.

Let us show that $G_{n,r}$ has an interval $(1+n(r-1))$-coloring.
%For that, 
We define an edge-coloring $\alpha$ of $G_{n,r}$ as
follows: first we color the edges from $E(u_{1}v_{1})$ with colors
$1,\ldots,r-1$ and from $E(u_{n}v_{n})$ with colors
$(n-1)(r-1)+2,\ldots,n(r-1)+1$; then we color the edges from
$E(u_{i}v_{i})$ with colors $(i-1)(r-1)+2,\ldots,i(r-1)$, where
$2\leq i\leq n-1$; finally we color the edges $u_{j}u_{j+1}$ and
$v_{j}v_{j+1}$ with color $j(r-1)+1$, where $1\leq j\leq n-1$. It is
straightforward to verify that $\alpha$ is an interval
$(1+n(r-1))$-coloring of $G_{n,r}$. Thus, $G_{n,r}\in \mathfrak{N}$
and $W\left(G_{n,r}\right)\geq 1+n(r-1)$. On the other hand, by
Corollary \ref{mycorollary3.1}, we have $W\left(G_{n,r}\right)\leq
1+n(r-1)$, so $W\left(G_{n,r}\right)= 1+n(r-1)$. 
\end{proof}
\bigskip

\section{Bounds on $w(G)$}\

In this section we prove some bounds on the mimimum number of colors
in an interval coloring of a multigraph $G$. Hanson and Loten
\cite{HansonLoten} proved a lower bound on the number of colors in
an interval coloring of an $(a,b)$-biregular graph. Here we give a
similar bound.

\begin{theorem}
\label{mytheorem4.1} If $G$ is a multigraph and $G\in
\mathfrak{N}$, then
\begin{center}
$w(G)\geq \left\lceil \frac{\vert V(G)\vert}{2\cdot
\alpha'(G)}\right\rceil\delta(G)$.
\end{center}
\end{theorem}

\begin{proof} Let $\delta=\delta(G)$. Consider an interval $w(G)$-coloring
$\alpha $ of $G$. In the coloring $\alpha$ of $G$, we consider the
edges with colors $\delta,2\delta,\ldots,k\delta$, where $k$ is the
maximum integer for which there exists an edge with color $k\delta$.
Clearly, $k\delta\leq w(G)$. Since each vertex $v$ of $G$ is
incident to at least one of the edges with color
$\delta,2\delta,\ldots,k\delta$, we have

\begin{center}
$\vert V(G)\vert\leq {\sum\limits_{i=1}^{k}2\vert
M_{\alpha}(i\cdot\delta)\vert}\leq 2k\cdot\alpha'(G)$,
\end{center}
where $M_{\alpha}(i\cdot\delta)=
\{e \in E(G) : \alpha(e)=i\cdot\delta\}$, $1\leq i\leq k$.

This implies that $k\geq \left\lceil \frac{\vert V(G)\vert}{2\cdot
\alpha'(G)}\right\rceil$. On the other hand, since $k\delta\leq w(G)$,
we obtain $w(G)\geq \left\lceil \frac{\vert V(G)\vert}{2\cdot
\alpha'(G)}\right\rceil\delta(G)$. %~$\square$
\end{proof}

\begin{figure}[h]
\begin{center}
\includegraphics[width=20pc]{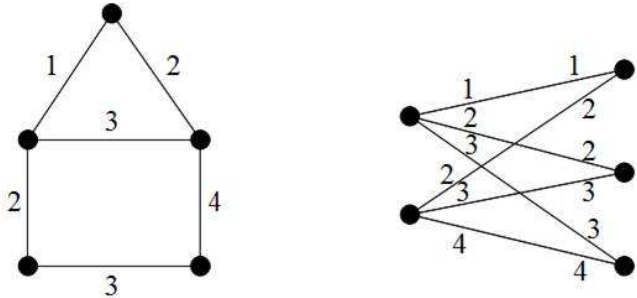}\\
\caption{Graphs with maximum degree $3$ without an interval
$3$-coloring.}\label{fig1}
\end{center}
\end{figure}

The graphs in Figure \ref{fig1} show that the lower bound in Theorem
\ref{mytheorem4.1} is sharp. Note also that there are infinite
families of graphs, such as all $(a-1,a)$-biregular graphs, 
for which the lower bound in
Theorem \ref{mytheorem4.1} is sharp.

\begin{corollary}
\label{mycorollary4.1} If a %connected 
multigraph $G$ has no perfect
matching and $G\in \mathfrak{N}$, then
\begin{center}
$w(G)\geq \max\{\Delta(G),2\delta(G)\}$.
\end{center}
\end{corollary}

Using similar counting arguments we can prove that Eulerian
multigraphs with an odd number of edges do not have interval
colorings. We first prove a more general theorem.

\begin{theorem}
\label{mytheorem4.2} If for a multigraph $G$, there exists a number
$d$ such that $d$ divides $d_{G}(v)$ for every $v\in V(G)$ and $d$
does not divide $\vert E(G)\vert$, then $G\notin \mathfrak{N}$.
\end{theorem}
\begin{proof}
Suppose, to the contrary, that $G$ has an interval $t$-coloring
$\alpha$ for some $t\geq \Delta(G)$. Let $d$ be a number such that $d$ divides $d_{G}(v)$ for
every $v\in V(G)$. We call an edge $e\in E(G)$ a \emph{$d$-edge} if $\alpha(e)=d\cdot l$ for some $l\in \mathbb{N}$.
Since $\alpha$ is an interval coloring, we have
that for any $v\in V(G)$, the set $S\left(v,\alpha\right)$ contains
exactly $\frac{d_{G}(v)}{d}$ $d$-edges. Now let $m_{d}$ be the
number of $d$-edges in $G$. By the Handshaking lemma, we obtain
$m_{d}=\frac{1}{2}\sum\limits_{v\in
V(G)}\frac{d_{G}(v)}{d}=\frac{\vert E(G)\vert}{d}$. Hence, $d$
divides $\vert E(G)\vert$, which is a contradiction. % ~$\square$
\end{proof}

\begin{corollary}
\label{mycorollary4.2} If $G$ is an Eulerian multigraph and $\vert E(G)\vert$ is odd, then $G\notin \mathfrak{N}$.
\end{corollary}

\begin{figure}[h]
\begin{center}
\includegraphics[width=20pc]{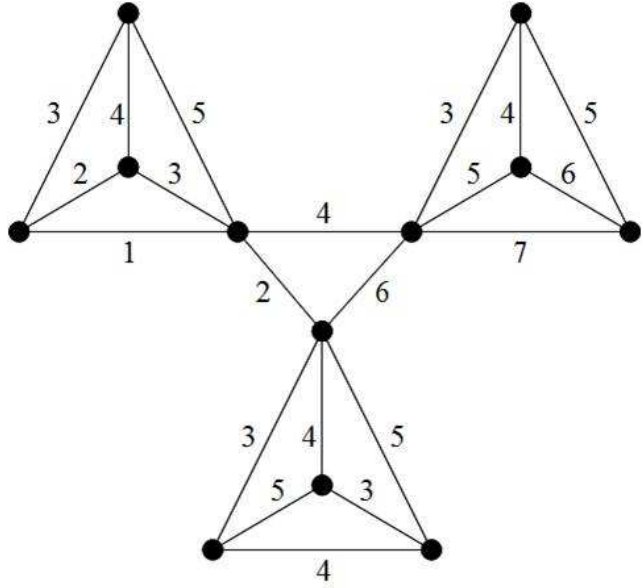}\\
\caption{An interval colorable connected odd graph $G$ with $\vert E(G)\vert-\frac{\vert V(G)\vert}{2}=15$ and $w(G)\geq
6$.}\label{fig2}
\end{center}
\end{figure}

Finally, we prove an analogue of Corollary \ref{mycorollary4.1} for odd multigraphs.

\begin{theorem}
\label{mytheorem4.3} If $G$ is an odd multigraph, $\vert
E(G)\vert-\frac{\vert V(G)\vert}{2}$ is odd and $G\in \mathfrak{N}$,
then
\begin{center}
$w(G)\geq \max\{\Delta(G),2\delta(G)\}$.
\end{center}
\end{theorem}
\begin{proof} Let $\delta=\delta(G)$. If $G$ has no perfect matching, then the result
follows from Corollary \ref{mycorollary4.1}. Assume that $G$ has a
perfect matching. Next suppose, to the contrary, that $G$ has an
interval $t$-coloring $\alpha$ for some $t\leq 2\delta-1$. Since for
every $v\in V(G)$, $1\leq \underline S\left(v,\alpha \right)\leq
\delta$, we obtain that $\delta\in \bigcap\limits_{v\in V(G)}
S\left(v,\alpha\right)$. This implies that the edges with color
$\delta$ form a perfect matching in $G$. Let $M$ be this perfect
matching. Consider the multigraph $G-M$. We define an edge-coloring
$\beta$ of $G-M$ as follows: for every $e\in E(G-M)$, let

\begin{center}
$\beta\left(e\right)=\left\{
\begin{tabular}{ll}
$\alpha(e)$, & if $1\leq \alpha(e)\leq \delta -1$,\\
$\alpha(e)-1$, & if $\delta+1\leq \alpha(e)\leq t$.\\
\end{tabular}%
\right.$
\end{center}

It is not difficult to see that $\beta$ is an interval
$(t-1)$-coloring of $G-M$. Since $\vert E(G)\vert-\frac{\vert
V(G)\vert}{2}$ is odd, we obtain that $G-M$ is an even multigraph
with an odd number of edges. This implies that $G-M$ has an Eulerian
component with an odd number of edges which is interval colorable,
but this contradicts Corollary \ref{mycorollary4.2}. 
\end{proof}

The graph in Figure \ref{fig2} shows that the lower bound in the
preceding theorem is sharp.

\bigskip

As we have seen, there are several lower bounds on $w(G)$ for different families of interval colorable graphs. In particular,
since a complete bipartite graph $K_{a,b}$ requires at least $a+b-\gcd(a,b)$ colors for an interval coloring
\cite{KamalianPreprint,KamalianThesis,Hansen}, it holds that for any positive integer $d$, there is a graph $G$ such that $G\in
\mathfrak{N}$ and $w(G) - \chi'(G)\geq d$. It is not known if a similar result holds for cyclic interval colorings. We would like to suggest the following.

\begin{problem}
    For any positive integer $d$, is there a graph $G$ such that $G\in \mathfrak{N}_{c}$ and $w_c(G)-\chi^{\prime}(G)\geq d$?
\end{problem}

Even the case $d=1$ of the above problem is open. The problem has a positive answer if $d=1$ and $\chi^{\prime}(G)=\Delta(G)$
\cite{PetrosyanMkhit}. Another variant of the problem is obtained by replacing $\chi'(G)$ by $\Delta(G)$. The answer to this latter
problem is positive for multigraphs, but it is open for $d\geq 2$ in the case of graphs. For the case $d=1$, this reformulated problem has a positive answer for ordinary graphs (see e.g.
\cite{PetrosyanMkhit,AsratianCasselgrenPetrosyan}).

%%%%%%%%%%%%%%%%%%%%%%%%%%%%%%%%%%%%%%%

\end{document}